\documentclass{amsart}
\usepackage[utf8]{inputenc}
\usepackage[english]{babel}
\usepackage{cite}
\usepackage{enumerate}
\usepackage{amsfonts,amsmath,amsthm,amssymb}
\usepackage{hyperref}
\usepackage{graphicx,multicol,xypic,xcolor}
\usepackage{mathtools}
\usepackage{tikz}
\usetikzlibrary{trees}

\newtheorem{theorem}{Theorem}[section]
\newtheorem{proposition}[theorem]{Proposition}
\newtheorem{lemma}[theorem]{Lemma}
\newtheorem{corollary}[theorem]{Corollary}

\newtheorem{example}[theorem]{Example}

\newcommand{\N}{\mathcal N}
\renewcommand{\P}{\mathcal P}

\newcommand{\EE}{\mathbb E}
\newcommand{\NN}{\mathbb N}
\newcommand{\PP}{\mathbb P}

\newcommand{\RR}{\mathbb R}

\DeclareMathOperator{\supp}{supp}

\DeclarePairedDelimiter\floor{\lfloor}{\rfloor}

\begin{document}

\title[Quasi-Assouad Dimensions for Random Measures]
      {Quasi-Assouad dimensions for random measures supported on $[0,1]^d$}

\author[W. Shen]{\textbf{Wanchun Shen} \\
Department of Pure Mathematics\\
University of Waterloo\\
Waterloo, ON, Canada \\
(w35shen@edu.uwaterloo.ca)}

\label{firstpage}
\maketitle 

\begin{abstract}
We introduce a probability distribution on $\P([0,1]^d)$, the space of all Borel probability measures on $[0,1]^d$. Under this distribution, almost all measures are shown to have infinite upper quasi-Assouad dimension and zero lower quasi-Assouad dimension (hence the upper and lower Assouad dimensions are almost surely infinite or zero). We also indicate how the results extend to other Assouad-like dimensions.
\end{abstract}




\section{Introduction} \label{sec:intro}
The \emph{Assouad dimension} of a metric space was first introduced by Assouad \cite{Assouad-A-dim, Assouad-A-dim2} in relation to the study of embedding problems. It was observed in \cite{CWY-A-star-dim} to coincide with the \emph{star dimension} introduced by Furstenberg \cite{FurstenbergExposition}.

Variants of the Assouad dimension, the \emph{lower Assouad dimension} \cite{L-lower-A-dim}, \emph{quasi-Assouad dimensions} \cite{LX-qA-dim} and \emph{$\Phi$-dimensions} \cite{GHM-phi-dim} (which includes quasi-Assouad dimensions and \emph{$\theta$-dimensions} \cite{FY-theta-dim} as special cases), were defined and studied by various authors. The analogous \emph{Assouad dimension} for measures, known also as the \emph{regularity dimension}, was first studied in \cite{KLV-A-dim-meas, KLV-A-dim-meas2, FH-upper-reg-dim}. Analogs of \emph{quasi-Assouad dimension} for measures were introduced and investigated in \cite{HHT-qA-dim-meas, HS-lower-qA-dim-meas}.

These (quasi)-Assouad and Assouad-like dimensions aim to capture the local extreme behavior of a set or a measure, and as such they are not necessarily finite (or bounded below from zero). Indeed, it was observed in \cite{FH-upper-reg-dim} that a measure has finite uppper Assouad dimension if and only if it is \emph{doubling}. The notion of \emph{quasi-doubling} was introduced in \cite{HHT-qA-dim-meas} to characterize measures with finite upper quasi-Assouad dimension.

In this note, we study the quasi-Assouad dimensions of randomly selected measures supported on $[0,1]^d$. We begin in Section \ref{sec:star-char} by defining quasi-Assouad dimensions and reformulate them in a way similar to Furstenberg's definition of star dimension. In Section \ref{sec:distn-on-meas}, we construct a Borel probability distribution on $\P([0,1]^d)$, the space of all probability measures on $[0,1]^d$, regarded as a (compact) metric space with respect to the weak$^*$ topology. This distribution is obtained by representing measures in terms of labeled trees, and then choosing the ratios between each level in a uniform manner. Because of the uniformity in our construction, the expectation of the measure of any Borel set agrees with the $d$-dimensional Lebesgue measure. In Section \ref{sec:qA-dim-infinite}, we prove that under this distribution, almost all measures have infinite upper quasi-Assouad dimension. Similarly, but with a slightly different estimate, the lower quasi-Assouad dimension of a measure is almost surely zero. In Section \ref{sec:extensions}, we indicate how these results can be extended to intermediate Assouad-like dimensions. We end with a list of questions that we have not yet been able to answer.

We mention that the ``generic" behavior (in both the probabilistic and topological sense) of Assouad dimensions of attractors of random iterated function systems have been studied before in \cite{FMT-A-dim-random-fractals, FT-A-dim-random-carpets} as well as other papers. However, the literature on the probabilistic behavior of Assouad dimensions for measures is scarce. In view of our results, rather than studying the behavior of the Assouad dimensions of general probability measures, one might want to focus on measures that are more ``structured", such as those with some self-similarity.

\section{Quasi-Assouad dimension for measures and their star characterizations} \label{sec:star-char}

Throughout, $\P([0,1]^d)$ will denote the space of Borel probability measures supported on $[0,1]^d\subset \RR^d$.

For $\mu\in\P([0,1]^d),\delta>0$, let 
\[H(\mu, \delta)=\inf\left\{s: \exists C>0, \, \forall 0<r<R^{1+\delta}<R<1\, , \sup_{x\in \supp\mu} \frac{\mu(B(x,R))}{\mu(B(x,r))}\leq C\left(\frac{R}{r}\right)^s \right\},\]
and
\[h(\mu, \delta)=\sup\left\{s: \exists C>0, \, \forall 0<r<R^{1+\delta}<R<1\, , \inf_{x\in \supp\mu} \frac{\mu(B(x,R))}{\mu(B(x,r))}\geq C\left(\frac{R}{r}\right)^s\right\}.\]

Here the balls are understood to be with respect to the Euclidean metric on $\RR^d$. It makes no difference in the definition if we use any other equivalent metric. In particular, we are allowed to use the supremum metric, with respect to which all balls are ``hypercubes" in the Euclidean sense.

The \emph{upper quasi-Assouad dimension} of $\mu$ is defined to be
\[\dim_{qA} \mu = \lim_{\delta\rightarrow 0} H(\mu,\delta),\] 
and the \emph{lower quasi-Assouad dimension}
\[\dim_{qL} \mu = \lim_{\delta\rightarrow 0} h(\mu,\delta).\] 

Note that as $\delta\rightarrow 0$, $H(\mu,\delta)$ increases and $h(\mu,\delta)$ decreases. Thus the limits exist, and we may take limits along any countable sequence $\delta_n$ decreasing to zero.

The \emph{upper Assouad dimension} is $\dim_A \mu = H(\mu,0)$, and the \emph{lower Assouad dimension} is $\dim_L \mu = h(\mu,0)$. It is not hard to see $\dim_{L} \mu \leq \dim_{qL} \mu \leq \dim_{qA} \mu \leq \dim_A \mu$.

Following Furstenberg, we say a set $E'\subset [0,1]^d$ is a \emph{mini-set} of $E\subset [0,1]^d$ if there exists a scalar $\lambda \geq 1$ and $u\in \RR^d$ such that $E' = (\lambda E + u) \cap [0,1]^d$. Intuitively, $E'$ is the magnification of a certain part of $E$. Similarly, a \emph{mini-measure} of $\mu \in \P([0,1]^d)$ is a norm-one rescaling of $f^*\mu$, i.e. a measure of the form $\frac{1}{f^* \mu([0,1]^d)} f^* \mu|_{[0,1]^d}$. Here $f^*\mu$ denotes the push-forward of $\mu$ by some affine map on $\RR^d$ of the form $f(x)=\lambda x + u$, $\lambda\geq 1$, $u\in \RR^d$. A mini-measure of $\mu$ is supported on some mini-set of $\supp \mu$. 

The \emph{star dimension} of $E\subset [0,1]^d$ is
\[\dim^{*}E = \lim_{m\rightarrow \infty} \frac{\log_2 H_m^* (E)}{m},\] 
where $H_m^*(E)$ is the maximum, over all mini-sets $E'$ of $E$, number of grids of length $2^{-m}$ that intersect $E'$. Note that the limit exists by sub-multiplicativity of $(H_m^*(E))_m$. We refer the readers to \cite{FurstenbergExposition} for more details of these well-known definitions and results.

It was shown in \cite{CWY-A-star-dim} that $\dim^*E = \dim_{A} E$. In a similar way, one can formulate a ``star characterization" of $H(\mu, \delta)$. To do so, we ``discretize" the condition $1<r<R^{1+\delta}<R$ by considering various $m\in \NN$ for which $2^m \approx \frac{R}{r} > \frac{R}{R^{1+\delta}} = R^{-\delta}$. Thus the mini-measures we consider must be supported on a rescaling of hypercubes of side length $l=2R>2\cdot 2^{-m/\delta}$.

More precisely, we define the \emph{upper $\delta$-star dimension} of $\mu$ as 
\[\dim_{\delta}^{*}\mu = \limsup_{m\rightarrow \infty} \frac{\log_2 H_{m,\delta}^*(\mu)}{m}.\]
To obtain $H_{m,\delta}^*(\mu)$, for every mini-measure $\mu'$ of $\mu$, temporarily define $M_{m,\delta}(\mu')$ to be the maximum $\mu'$ measure of the $2^d$ central hypercubes in the partition of $[0,1]^d$ into $2^{md}$ hypercubes each of side length $2^{-m}$. Let $H_{m,\delta}^*(\mu):= \sup \frac{1}{M_{m,\delta}(\mu')}$, where the supremum is taken over all mini-measures $\mu'$ that arise from rescaling of a hypercube of side length $l> 2 \cdot 2^{-\frac{m}{\delta}}$, for which $M_{m,\delta}(\mu')\neq 0$. 

Unravelling all relevant definitions, we have

\begin{lemma}
$H(\mu, \delta)=\dim_{\delta}^{*} \mu.$
\begin{proof}
We first show $\dim_{\delta}^* \mu \leq H(\mu,\delta)$. Take a sequence $\alpha_n \searrow H(\mu,\delta)$ such that for each $n$, there exists $C_n>0$ such that
\[\sup_{x\in \supp \mu} \frac{\mu(B(x,R))}{\mu(B(x,r))} \leq C_n (\frac{R}{r})^{\alpha_n},\]
whenever $0<r<R^{1+\delta}<R<1$.

For each $m\in \NN$, $\epsilon>0$, $H_{m,\delta}^*(\mu) \leq \frac{1}{M_{m,\delta}(\mu')}+\epsilon = \frac{1}{\mu'(B)} +\epsilon$ for some mini-measure $\mu'$ of $\mu$, whose support is a rescaling of a hypercube of side length $l\geq 2\cdot 2^{-\frac{m}{\delta}}$. Here $B$ is one of the central $2^d$ hypercubes with the largest measure. Since this condition on $l$ is equivalent to $\frac{1}{2} l \cdot 2^{-m} < (\frac{1}{2}l)^{1+\delta}$, for some suitable $x,x'\in [0,1]^d$, we have
\[H_{m,\delta}^* (\mu) \leq \frac{1}{\mu'(B)} + \epsilon \leq \frac{1}{\frac{1}{2^d} \mu'(x,2^{-m-1})}+\epsilon = \frac{2^d \mu(B(x',l/2))}{\mu(B(x',l/2 \cdot 2^{-m}))} +\epsilon \leq 2^d C_n 2^{m\alpha_n} +\epsilon,\]
for each $n$ and every $\epsilon>0$.

It follows, by letting $\epsilon\rightarrow 0$, $m\rightarrow\infty$ and $n\rightarrow\infty$, that
\[\dim_{\delta}^* \mu \leq H(\mu,\delta).\]

To prove the other direction, fix $\epsilon > 0$. For simplicity of notation, let $s:=\dim^*_{\delta} \mu$. By the definition of the $\delta$-star dimension, for large enough $m$, we have $H_{m,\delta}^*(\mu) \leq 2^{m(s + \epsilon)}$. Thus there exists $C>0$ (depending only on $\epsilon$) such that for all $m\in \NN$, 
\[H_{m,\delta}^*(\mu) \leq C \cdot 2^{m(s + \epsilon)}.\]

Given $0<r<R^{1+\delta}<R<1$, choose $m$ such that $2^{m-1} \leq \frac{R}{r} < 2^m$. For all $x\in \supp \mu$,
\[\frac{\mu(B(x,R))}{\mu(B(x,r))} \leq \frac{\mu(B(x,R))}{\mu(B(x,2^{-m} R))} \leq  \frac{1}{M_{m,\delta}(\mu')} \leq H_{m,\delta}^*(\mu) \leq C \cdot 2^{m(s+\epsilon)} \leq C' (\frac{R}{r})^{s+\epsilon},\]
where $R\geq 2^{-\frac{m}{\delta}}$ is justified by $R^{-\delta}=\frac{R}{R^{1+\delta}} \leq \frac{R}{r}<2^m$.

Thus,
\[\sup_{x\in \supp \mu} \frac{\mu(B(x,R))}{\mu(B(x,r))} \leq C' \cdot \left(\frac{R}{r}\right)^{s + \epsilon}.\]
It follows that $H(\mu,\delta) \leq s + \epsilon$ for all $\epsilon>0$. Hence
\[H(\mu,\delta) \leq \dim_{\delta}^* \mu.\]
\end{proof}
\end{lemma}

The analogous statement for the lower quasi-Assouad dimension also holds. We omit the proof, as it is similar.

\begin{lemma}
We have
\[h(\mu, \delta) = \liminf_{m\rightarrow\infty} \frac{\log_2 h_{m,\delta}^*(\mu)}{m},\]
where $h_{m,\delta}^*(\mu)$ is the minimum, over all mini-measures of $\mu'$ that are rescaling of some hypercubes of side length $l> 2 \cdot 2^{-\frac{m}{\delta}}$, of the reciprocal of the largest $\mu'$-measure of the central $2^d$ hypercubes from the partition of $\supp(\mu')$ into $2^{md}$ hypercubes (of side length $2^{-m}$).
\end{lemma}

\section{A Distribution on Probability measures}
\label{sec:distn-on-meas}
The goal of this Section is to construct a Borel probability distribution $\PP$ on $\P([0,1]^d)$, regarded as a (compact) metric space with the weak$^*$ topology. We first discuss the represention of $\mu\in \P([0,1]^d)$ in terms of a labeled $2^d$-ary tree, then construct a distribution on the space of labeled trees by choosing the ratios between each level in a uniform manner. Finally, we justify that it is a Borel probability distribution with respect to which measures behave like the Lebesgue measure ``on average".

\subsection{Tree representation of a probability measure} 
It is well-known that there is a 1-1 correspondence between closed subsets in $[0,1]^d$ and $2^d$-ary trees, and that measures on $[0,1]^d$ can be represented in terms of labeled $2^d$-ary trees, c.f.\cite{FurstenbergExposition}. For measures without full support, the corresponding tree is not necessarily complete (i.e. there might be some nodes with $<2^d$ children). However, we may complete it and label by zero all nodes that we add. Henceforth, when we speak of $n$-ary labeled trees, we mean a complete tree, possibly with some nodes (and of course, all their children) labeled zero.

We also label each edge of the tree by the ratio of the labels of the child node to the parent node.

For example, when $d=1$, the tree representation of $\mu\in \P([0,1])$ is

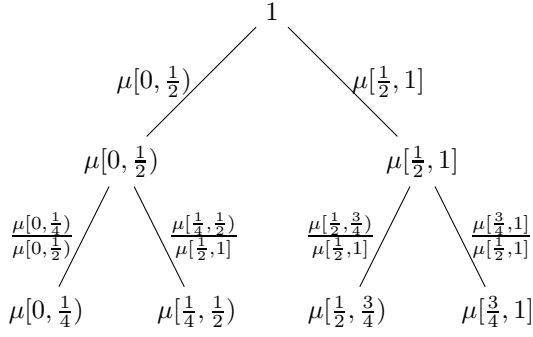
\begin{tikzpicture} 
[level distance=2cm,
level 1/.style={sibling distance=4cm},
level 2/.style={sibling distance=2cm}]

  \node {1} 
    child { 
      node {$\mu[0,\frac{1}{2})$} 
        child {
          node {$\mu[0,\frac{1}{4})$}
          edge from parent node[left] {$\frac{\mu[0,\frac{1}{4})}{\mu[0,\frac{1}{2})}$}        
        }
        child {
          node {$\mu[\frac{1}{4},\frac{1}{2})$}
          edge from parent node[right] {$\frac{\mu[\frac{1}{4},\frac{1}{2})}{\mu[\frac{1}{2},1]}$}
        }
      edge from parent node[left] {$\mu[0,\frac{1}{2})$} 
    } 
    child { 
      node {$\mu[\frac{1}{2},1]$} 
        child { 
          node {$\mu[\frac{1}{2},\frac{3}{4})$} 
          edge from parent node[left] {$\frac{\mu[\frac{1}{2},\frac{3}{4})}{\mu[\frac{1}{2},1]}$} 
        } 
        child {
          node {$\mu[\frac{3}{4},1]$}
          edge from parent node [right] {$\frac{\mu[\frac{3}{4},1]}{\mu[\frac{1}{2},1]}$}
          } 
      edge from parent node[right] {$\mu[\frac{1}{2},1]$} 
    }; 
\end{tikzpicture}

The label on each node is then the product of the labels on all edges of the (unique) path joining the root and this node. Note also that the sum of the labels of all edges below every node is 1. Thus, if we label all edges in a tree subject to the condition that all edge labels below each node sum to 1, then all labels on the nodes are uniquely determined, and correspond to a measure supported on $[0,1]^d$.

\subsection{Construction of the distribution}
Let $n\geq 2$ be an integer. In the following, we will construct a distribution on  labeled $n$-ary trees by choosing, independently for each node, $n$ random nonnegative numbers adding up to one, and arranging them to be the labels on the edges below this node.

Let $X_0=(X_1,...,X_n)$ be a random vector that takes values uniformly in the $(n-1)$-simplex $\left\{(x_1,...,x_n)\in \RR^n: x_1,...,x_n \geq 0, x_1+...+x_n=1\right\}$, i.e. its probability density is constant with respect to the $(n-1)$-dimensional Lebesgue measure\footnote{This is known in the literature as a special case of the \emph{Dirichlet distribution}.}. It is well-known (or follows from a simple calculation) that the marginal distribution of each $X_i$ is the beta distribution Beta(1,$n-1$). That is, $\PP_{X_i}(X_i\leq c) = 1-(1-c)^{n-1}$ for $c\in [0,1]$.

For convenience, we index the nodes of an $n$-ary tree by sequences with digits in $\Lambda = \{1,...,n\}$. Let $\Omega = \bigcup_n \Lambda^n$ be the semigroup generated by $\Lambda$, the elements of which will represent nodes of our $n$-ary trees. Thus every labelled tree can be identified as a point in $[0,1]^{\Omega}$. 

By a slight abuse of notation, instead of ``a node of the labeled tree $T$ represented by $\omega$", we will simply say ``the node $\omega$". Sometimes, we will also denote the label of the node $\omega$ in $T$ by $m_T(\omega)$.

Next, we make a countable collection of independent copies of $X_0$, and attach to each node $\omega\in \Omega$ such an independent copy $X^{\omega}$. The $n$ edges below a node $\omega$ will be attached the corresponding components of $X^\omega = (X^\omega_1,...,X^\omega_n)$. As remarked above, the law of $X:=(X^{\omega})_{\omega\in \Omega}$ determines a distribution on the space of labeled $n$-ary trees, which we shall denote by $\PP$. The underlying $\sigma$-algebra will be denoted by $\sigma(X)$. 

Since $\P([0,1]^d)$ can be identified with the space of labeled $2^d$-ary trees, applying the construction to $n=2^d$, we obtain a distribution on $\P([0,1]^d)$. By abuse of notation, we still denote it by $\PP$, and its underlying $\sigma$-algebra $\sigma(X)$.

We illustrate the situation with $d=1$. In this case every $\mu\in \P([0,1])$ can be represented by a labeled binary tree, where each $X_{\omega}$ (with $\omega$ an 0-1 sequence) has distribution Beta$(1,1)$, or equivalently, Uniform$(0,1)$:

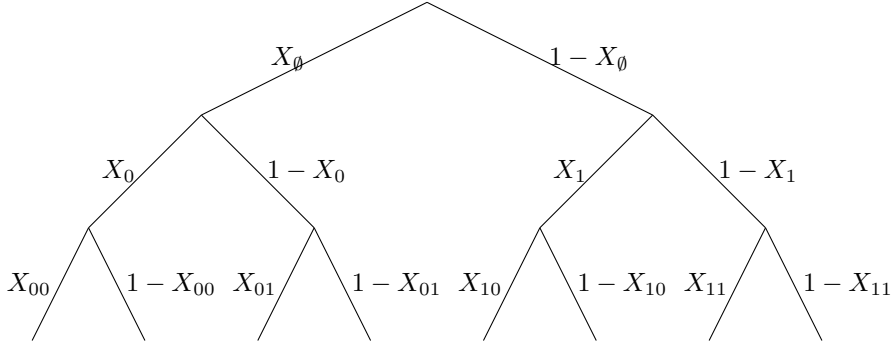
\begin{tikzpicture} 
[level distance=1.5cm,
level 1/.style={sibling distance=6cm},
level 2/.style={sibling distance=3cm},
level 3/.style={sibling distance=1.5cm}]

\coordinate
  child { 
    child {
      child{edge from parent node[left] {$X_{00}$}}
      child{edge from parent node[right] {$1-X_{00}$}}
      edge from parent node[left] {$X_{0}$}        
        }
    child {
      child{edge from parent node[left] {$X_{01}$}}
      child{edge from parent node[right] {$1-X_{01}$}}
      edge from parent node[right] {$1-X_{0}$}
        }
    edge from parent node[left] {$X_{\emptyset}$} 
  } 
  child { 
    child {
      child{edge from parent node[left] {$X_{10}$}}
      child{edge from parent node[right] {$1-X_{10}$}} 
      edge from parent node[left] {$X_1$} 
      } 
    child {
      child{edge from parent node[left] {$X_{11}$}}
      child{edge from parent node[right] {$1-X_{11}$}}
      edge from parent node [right] {$1-X_1$}
      } 
      edge from parent node[right] {$1-X_{\emptyset}$} 
  }; 
\end{tikzpicture}

Enumerating elements of $\Omega$ in the order $\emptyset, 0,1,00,01,10,11,...$, we have a random variable whose value specifies all edge labels
\[X = ((X_{\emptyset},1-X_{\emptyset}),(X_0,1-X_0),(X_1,1-X_1),(X_{00},1-X_{00}),...).\]

This uniquely determine all node labels, namely, the value of 
\[Y = (1, X_{\emptyset}, 1 - X_{\emptyset}, X_{\emptyset}X_0,X_{\emptyset}(1-X_0),(1-X_{\emptyset})X_1, (1-X_{\emptyset})(1-X_1), X_{\emptyset}X_0 X_{00}, ...).\]

Hence, they induce a distribution on the space of labeled binary trees, and on $\P([0,1])$, via the push-forward.

\subsection{Basic properties of the distribution}
\begin{lemma}
The distribution $\PP$ on $\P([0,1]^d)$ constructed above is a Borel probability measure.
\begin{proof}
It is clear that $\PP$ is a probability measure, being the push-forward of some probability measure on the sample space of $X$. It remains to show all weak$^*$ neighbourhoods of $\P([0,1]^d)$ are $\PP$-measurable.

It follows, by approximating all continuous functions on $[0,1]^d$ by linear combinations of characteristic functions on $2$-adic hypercubes, that the push-forward of the weak$^*$ topology on $\P([0,1]^d)$ corresponds to the topology defined by pointwise convergence of the labels on every node of the labeled trees. A subbase of this topology consists of elements of the form
\[\N(T_0, \omega,\epsilon) = \left\{T: |m_T(\omega)-m_{T_0}(\omega)|<\epsilon\right\},\]
for some fixed tree $T_0$, finite sequence $\omega\in \Omega$ and $\epsilon>0$. Recall that the notation $m_T(\omega)$ stands for the label of $\omega$ on $T$. Since $m_T(\omega)$ (regarded as a function on $T$) is given by the product of some $X^\omega_i$, it is measurable. It follows that $\N(T_0, \omega,\epsilon)$ is measurable and $\PP$ is a Borel measure.
\end{proof}
\end{lemma}

Since our random measures are selected ``uniformly" at each level, we should expect them to behave probabilistically like the Lebesgue measure

\begin{lemma}
Let $E\subset [0,1]^d$ be a Borel set. Then the expectation satisfies
\[\EE[\mu(E)] = m(E),\]
where $m$ denotes the $d$-dimensional Lebesgue measure.

\begin{proof}
First, suppose $E$ is a 2-adic hypercube with side length $2^{-m}$. Then $\mu(E)$ is, by construction, the product of $m$ independent random variables with expectation $2^{-d}$. (Recall that they have the beta distribution with parameters 1 and $2^d-1$.) Thus $\mu(E)$ has expectation $2^{-md}$, the $d$-dimensional Lebesgue measure of $E$.

In the general case, we may approximate $E$ by linear combinations of 2-adic hypercubes. The result then follows from the linearity of expectation and continuity of measures.
\end{proof}

\end{lemma} 

\section{Almost-sure Behavior of quasi-Assouad Dimensions}
\label{sec:qA-dim-infinite}

\subsection{Non-Finiteness of the upper quasi-Assouad dimension} \label{sec:finiteness-qA-dim}
The goal of this subsection is to show
\[\dim_{qA} \mu = \lim_{\delta\rightarrow 0} H(\mu,\delta)=\lim_{\delta\rightarrow 0} \limsup_{m\rightarrow\infty} \frac{\log_2 H^*_{m,\delta}(\mu)}{m}\]
is infinite, for $\PP$-a.e. $\mu$.

To do this, we introduce a new quantity, $H_{m,\delta}(\mu)$, closely related to $H_{m,\delta}^*(\mu)$, which works more naturally with the tree representation. Fix $m,\delta$. Each node at level $\floor{\frac{m}{\delta}}$ has $2^m$ children $m$ levels down, $2^d$ of which correspond to the central $2^d$ hypercubes in the partition of $[0,1]^d$ into $2^{\floor{m/\delta}+m}$ hypercubes. We refer to these as \emph{central children}. Write $|\omega|$ for the length of the (finite) sequence $\omega$. Define 
\[H_{m,\delta}(\mu) = \max_{|\omega| = \floor{\frac{m}{\delta}}} \min_{\substack{\omega' \text{ a central} \\ \text{child of } \omega}} \frac{m_T(\omega)}{m_T(\omega')}.\]

Continuing with our example in the case $d=1$,
\[H_{2,2}(\mu) = \max\left\{\min\left\{X_0(1-X_{00}),(1-X_0)X_{01}\right\}, \min\left\{X_1(1-X_{10}),(1-X_1)X_{11}\right\}\right\}.\]

Since this amounts to taking the supremum over a smaller collection of mini-measures, $H_{m,\delta}(\mu) \leq H_{m,\delta}^*(\mu)$. Thus it suffices to show for $\PP$-a.e. $\mu$,
\[H(\mu):=\lim_{\delta\rightarrow 0} \limsup_{m\rightarrow\infty} \frac{\log_2 H_{m,\delta}(\mu)}{m} = \infty.\]

We begin by showing the set in question belongs to the tail $\sigma$-algebra
\[\bigcap_{N\in \NN} \sigma(X^{\omega}: |\omega| \geq N),\]
Here, since the image of each $X^{\omega}$ embeds in the space of trees, by abuse of notation we may think of the $\sigma$-algebra $\sigma(X^{\omega})$ generated by $X^{\omega}$ as a subset of $\sigma(X)$.

\begin{lemma} \label{lemma:tail-event}
$H(\mu) = \infty$ is a tail event. That is,
\[\left\{\mu: \lim_{\delta\rightarrow 0} \limsup_{m\rightarrow\infty} \frac{\log_2 H_{m,\delta}(\mu)}{m}<\infty\right\} \in \bigcap_{N\in \NN} \sigma(X^{\omega}: |\omega| \geq N).\]
Thus, by Kolmogorov's 0-1 law, this set has $\PP$-measure 0 or 1.
\end{lemma}

\begin{proof}
We claim that for each $\delta>0$,
\[\left\{\mu: \limsup_{m\rightarrow\infty} \frac{\log_2 H_{m,\delta}(\mu)}{m}<\infty\right\}\]
is a tail event. Indeed, by definition $H_{m,\delta}(\mu)$ depends only on $X^{\omega}$ with $|\omega|\geq \floor{\frac{m}{\delta}}$. Thus for each $N$, let $M$ be so large such that $\floor{\frac{M}{\delta}}\geq N$, then
\begin{align*}
\left\{\mu: \limsup_{m\rightarrow\infty} \frac{\log_2 H_{m,\delta}(\mu)}{m}<\infty\right\} 
&= \left\{\limsup_{m\rightarrow\infty, \, m\geq M} \frac{\log_2 H_{m,\delta}(\mu)}{m}<\infty\right\} \\
&\in \sigma(X^{\omega}: |\omega|\geq N).
\end{align*}

Since $\limsup_{m\rightarrow\infty} \frac{\log_2 H_{m,\delta}(\mu)}{m}$ increases as $\delta$ decreases to zero, we may take limits along a countable subsequence (which preserves measurability) to see
\[\left\{\mu: \lim_{\delta\rightarrow 0} \limsup_{m\rightarrow\infty} \frac{\log_2 H_{m,\delta}(\mu)}{m}<\infty\right\}\]
is also a tail event.
\end{proof}

\begin{proposition} \label{prop:delta-dim-is-small-with-small-prob}
For all $N\in \NN$, there exists $\delta>0$ such that
\[\PP\left\{\mu: \limsup_{m\rightarrow\infty} \frac{\log_2 H_{m,\delta}(\mu)}{m}\leq N\right\} \leq e^{-1}.\]

\begin{proof}
We first note that $$\left\{\mu: \limsup_{m\rightarrow\infty} \frac{\log_2 H_{m,\delta}(\mu)}{m}\leq N\right\}$$ is $\PP$-measurable as $\limsup_{m\rightarrow\infty} \frac{\log_2 H_{m,\delta}(\mu)}{m}$ can be obtained by taking limit, maximum, minimum, reciprocal and products of $X^{\omega}_i$, and all these operations preserve measurability.

For $N\in \NN$, we show that any $\delta< 1/(N 2^d (2^d-1))$ works. For every $m\in \NN$, we have that $\frac{\log_2 H_{m,\delta}(\mu)}{m}\leq N$ if and only if 
$$H_{m,\delta}(\mu) = \max_{|\omega| = \floor{\frac{m}{\delta}}} \min_{\substack{\omega'\text{ a central} \\ \text{child of }\omega}} \frac{m_T(\omega)}{m_T(\omega')} \leq 2^{Nm},$$
if and only if all $\min_{\omega'}\frac{m_T(\omega)}{m_T(\omega')}\leq 2^{Nm}$. Recall that $\frac{m_T(\omega')}{m_T(\omega)}$ is given by the product of $m$ independent random variables with distribution Beta(1,$2^d-1$). Among such products, those starting with distinct $\omega$ are independent (there are at least $2^{\floor{\frac{m}{\delta}}}$ of them). For $m\geq 2$, we have
\begin{align*}
\PP\left\{\mu: \frac{\log_2 H_{m,\delta}(\mu)}{m} \leq N\right\}
&= \PP\left\{\mu: H_{m,\delta}(\mu) \leq 2^{Nm}\right\}\\
&= \PP\left\{T: \max_{|\omega| = \floor{\frac{m}{\delta}}} \min_{\substack{\omega'\text{ a central} \\ \text{child of }\omega}} \frac{m_T(\omega)}{m_T(\omega')} \leq 2^{Nm}\right\} \\
&\leq \PP\left(\bigcap_{|\omega|=\floor{\frac{m}{\delta}}} \left\{ \max_{\omega'} X^{\omega}_{i_1} ... X^{\omega'}_{i_m} \geq 2^{-Nm}\right\} \right)\\
&\leq \PP\left(\bigcap_{|\omega|=\floor{\frac{m}{\delta}}} \left\{ \max_{\omega'} X^{\omega'}_{i_m} \geq 2^{-Nm}\right\}\right) \\
&= \left(1 - \PP\left\{ \max_{\substack{\omega' \text{ a central } \\ \text{child of some (fixed) } \omega_0}} X^{\omega'}_{i_m} < 2^{-Nm}\right\}\right)^{2^{\floor{\frac{m}{\delta}}}} \\
&= \left(1 - \left(\PP_Y\left\{Y < 2^{-Nm}\right\}\right)^{2^d}\right)^{2^{\floor{\frac{m}{\delta}}}},
\end{align*}

\noindent where $Y$ has distribution Beta$(1,2^d-1)$. The last equality is justified by the fact that each $\omega$ has $2^d$ central children, and for $m\geq 2$, the last terms in the corresponding products are independent.

Since $\delta$ has been chosen so that $\big(\PP_Y\left\{Y<2^{-Nm}\right\}\big)^{2^d} = 2^{-Nm(2^d-1)2^d}> 2^{-\floor{\frac{m}{\delta}}}$, using the fact that $\limsup_m (1-\frac{1}{a_m})^{b_m} \leq e^{-1}$ if $b_m\geq a_m$ eventually, we see that
\[\PP\left\{\mu: \limsup_{m\rightarrow\infty} \frac{\log_2 H_{m,\delta}(\mu)}{m}\leq N\right\} \leq \limsup_{m\rightarrow\infty} \PP\left\{\mu: \frac{\log_2 H_{m,\delta}(\mu)}{m} \leq N\right\} \leq  e^{-1},\]
as required.
\end{proof}
\end{proposition}

\begin{theorem} \label{thm:qA-dim-infinite-a-e}
Let $\PP$ be the distribution on $\P([0,1]^d)$ constructed in Section \ref{sec:distn-on-meas}. Then
\[\PP\left\{\lim_{\delta\rightarrow 0} \limsup_{m\rightarrow\infty} \frac{\log_2 H_{m,\delta}(\mu)}{m} <\infty\right\}=0.\]
In particular, the quasi-Assouad dimension is infinite for $\PP$-a.e. $\mu$.

\begin{proof}
By Proposition \ref{prop:delta-dim-is-small-with-small-prob}, for every $N$ there exists $\delta\geq 0$ such that
\[\PP\left\{\mu: \limsup_{m\rightarrow \infty} \frac{\log_2 H_{m,\delta}(\mu)}{m} \leq N\right\}\leq e^{-1}.\]

Since $\limsup_{m\rightarrow\infty} \frac{\log_2 H_{m,\delta}(\mu)}{m}$ increases as $\delta\searrow 0$, by monotonicity,
\[\PP\left\{\mu: \lim_{\delta\rightarrow 0} \limsup_{m\rightarrow\infty} \frac{\log_2 H_{m,\delta}(\mu)}{m} \leq N\right\}\leq e^{-1}<1.\]

By continuity of $\PP$,
\[\PP\left\{\mu: \lim_{\delta\rightarrow 0} \limsup_{m\rightarrow\infty} \frac{\log_2 H_{m,\delta}(\mu)}{m} <\infty\right\}\leq e^{-1}<1.\]

We saw in Lemma \ref{lemma:tail-event} that $\PP\left\{\mu: \lim_{\delta\rightarrow 0} \limsup_{m\rightarrow\infty} \frac{\log_2 H_{m,\delta}(\mu)}{m} <\infty\right\} = 0$ or 1, it must be the case that
\[\PP\left\{\mu: \lim_{\delta\rightarrow 0} \limsup_{m\rightarrow\infty} \frac{\log_2 H_{m,\delta}(\mu)}{m} <\infty\right\}=0.\]
\end{proof}

\end{theorem}

Since the upper Assouad dimension is an upper bound for the upper quasi-Assouad dimension, we have\footnote{This is also easy to see directly from the definition.}

\begin{corollary} \label{cor: A-dim-infinite}
Let $\PP$ be the distribution on $\P([0,1]^d)$ constructed in Section \ref{sec:distn-on-meas}. Then for $\PP$-a.e. $\mu$, $\dim_A \mu = \infty$.
\end{corollary}

\subsection{Non-Positivity of lower quasi-Assouad dimension}
The proof that the lower quasi-Assouad dimension is $\PP$-a.s. zero is similar.

Define 
$$h_{m,\delta}(\mu):= \min_{|\omega|=\floor{\frac{m}{\delta}}} \min_{\substack{\omega'\text{ a central} \\ \text{child of }\omega}} \frac{m_T(\omega)}{m_T(\omega')}.$$

Similarly, $h_{m,\delta}^*(\mu) \leq h_{m,\delta}(\mu)$, and thus to show the lower quasi-Assouad dimension is almost surely zero, it is enough to show for $\PP$-a.e. $\mu$,
\[\lim_{\delta\rightarrow 0} \liminf_{m\rightarrow\infty} \frac{\log_2 h_{m,\delta}(\mu)}{m} = 0. \]

As in Lemma \ref{lemma:tail-event}, the set in question is measurable and is a tail event.

\begin{proposition} \label{prop:lower-delta-dim-is-large-with-small-prob}
For all $N\in \NN$, there exists $\delta>0$ (depending on $N$) such that
\[\PP\left\{\mu: \liminf_{m\rightarrow\infty} \frac{\log_2 h_{m,\delta}(\mu)}{m}\geq \frac{1}{N}\right\} \leq e^{-1}.\]

\begin{proof}
Repeating the proof of Proposition \ref{prop:delta-dim-is-small-with-small-prob}, we see that for every $m$,
\begin{align*}
\PP\left\{\mu: \frac{\log_2 h_{m,\delta}(\mu)}{m} \geq \frac{1}{N}\right\}
&\leq \PP\left\{X_1...X_m\leq 2^{-m/N}\right\}^{2^{\floor{m/\delta}}} \\
&= \left(1-\PP\left\{X_1...X_m \geq 2^{-m/N}\right\}\right)^{2^{\floor{m/\delta}}},
\end{align*}

\noindent where the $X_i$'s are independent with distribution Beta($1,2^d-1$). By their independence,
\[\PP\left\{X_1...X_m \geq 2^{-m/N}\right\} \geq \PP\left\{X_1 \geq 2^{-1/N}\right\}^m = (1-2^{-1/N})^{(2^d-1) m}.\]

Taking $\delta$ small enough so that $(1-2^{-1/N})^{2^d-1} > 2^{-1/\delta}$, and using the fact that $\limsup_m (1-\frac{1}{a_m})^{b_m} \leq e^{-1}$ provided $b_m\geq a_m$ eventually, we have
\begin{align*}
\PP\left\{\mu: \liminf_{m\rightarrow\infty}\frac{\log_2 h_{m,\delta}(\mu)}{m} \geq \frac{1}{N}\right\} 
&\leq \limsup_{m\rightarrow\infty}\PP\left\{\mu: \frac{\log_2 h_{m,\delta}(\mu)}{m} \geq \frac{1}{N}\right\} \\
&\leq \limsup_{m\rightarrow\infty} \left(1-(1-2^{-1/N})^{(2^d-1)m}\right)^{2^{\floor{\frac{m}{\delta}}}}\\ 
&\leq  e^{-1}.
\end{align*}

The rest of the proof is then similar to Proposition \ref{prop:delta-dim-is-small-with-small-prob}.
\end{proof}
\end{proposition}

As in the argument for Theorem \ref{thm:qA-dim-infinite-a-e}, we have the following

\begin{theorem} \label{thm:qL-dim-zero-a-e}
Let $\PP$ be the distribution on $\P([0,1]^d)$ constructed in Section \ref{sec:distn-on-meas}. Then
\[\PP\left\{\mu: \lim_{\delta\rightarrow 0} \liminf_{m\rightarrow\infty} \frac{\log_2 h_{m,\delta}(\mu)}{m} >0\right\}=0.\]
In particular, the lower quasi-Assouad dimension, and hence the lower Assouad dimension, are zero for $\PP$-a.e. $\mu$.
\end{theorem}

\section{Extensions to Intermediate Assouad-like Dimensions}
\label{sec:extensions}
In this Section, we briefly indicate how our methods extend to the $\Phi$-dimensions introduced in \cite{GHM-phi-dim}. For simplicity, we focus only on the \emph{upper} $\Phi$-dimensions obtained by imposing the condition that $0<r<R^{1+\Phi(R)}<R<1$ in place of $0<r<R^{1+\delta}<R<1$ in the definition of quasi-Assouad dimension, where $\Phi: (0,1) \rightarrow \RR^+$ is a positive real-valued function. That is, we define $\dim_{\Phi} \mu$ to be
\[\inf\left\{s: \exists C>0, \, \forall 0<r<R^{1+\Phi(R)}<R<1\, , \sup_{x\in \supp\mu} \frac{\mu(B(x,R))}{\mu(B(x,r))}\leq C\left(\frac{R}{r}\right)^s\right\}.\]

We begin by finding a suitable *-characterization in this case: We want
\[R^{-\Phi(R)}=\frac{R}{R^{1+\Phi(R)}}<\frac{R}{r}\approx 2^m,\]
thus in $H_{m,\Phi}$, we consider only mini-measures whose support is a rescaling of a hypercube of side length $>R$, where $R$ satisfies $\varphi(R):=R^{-\Phi(R)}<2^m$.

\begin{example}
To recover the (upper) Assouad dimension, take $\Phi(R)=1/\log_2 R$, so that $\varphi(R)$ is a constant. Note that for large enough $m$, the condition $\varphi(R)<2^m$ is trivially satisfied, i.e. there are no constraint on $R$.
\end{example}

\begin{example}
To recover the function $H(\mu,\delta)$ in the definition of (upper) quasi-Assouad dimension, take $\Phi(R)=\delta$, so $\varphi(R)=R^{-\delta}$. In this case, the condition $\varphi(R)<2^m$ is equivalent to $R>2^{-m/\delta}$, which coincides with our previous definition in Section \ref{sec:star-char}.
\end{example}

Observe that $\varphi$ is nonincreasing in both cases. In the following, we assume $\varphi$ is nonincreasing and left-continuous. Define its \emph{quasi-inverse} to be $\varphi^{-1}(y):=\sup\left\{x\in (0,1): \varphi(x)\geq y\right\}$\footnote{This is an analog of the \emph{quantile function} in probablity theory. It follows from left continuity of $\varphi$ that the supremum can be replaced by maximum.}. It is characterized by the property that it is nonincreasing, right-continuous, and $\varphi(x)\geq y$ if and only if $x\leq \varphi^{-1}(y)$. 

Thus, our requirement $\varphi(R)<2^m$ on $R$ is equivalent to $R>\varphi^{-1}(2^m)=2^{-n}$, where $n=-\log_2 \varphi^{-1}(2^m)$.

Repeating the proof for the infiniteness of quasi-Assouad dimension, Proposition \ref{prop:delta-dim-is-small-with-small-prob}, we see that a sufficient condition for
\[\PP\left\{\limsup_{m\rightarrow\infty} \frac{\log_2 H_{m,\Phi}}{m}\leq N\right\}\leq e^{-1}\]

\noindent is that $(1-\frac{1}{2^{Nm}})^{2^{\floor{-\log_2 \varphi^{-1}(2^m)}}}$ be bounded away from 1 for $m$ sufficiently large. This is true provided $-\log_2 \varphi^{-1}(2^m)\geq Nm$, or equivalently,
\[\varphi(2^{-Nm})\leq 2^m\]
in view of the property of quasi-inverse.

This condition would be satisfied if $\varphi(x)\leq x^{-1/N}$, when $x$ is small enough.

\begin{example}
Let $\varphi(x)$ be constant. We see that for small enough $x$, the condition $\varphi(x)<x^{-1/N}$ is trivially satisfied. It follows that the upper Assouad dimension is infinite for $\PP$-a.e. $\mu$.
\end{example}

\begin{example}
Let $\varphi(x)=x^{-\delta}$. The condition $\varphi(x)<x^{-1/N}$ is satisfied for $0<x<1$ when $\delta\leq 1/N$. This is the conclusion of Proposition \ref{prop:delta-dim-is-small-with-small-prob}.
\end{example}

Recall $\varphi(x)=x^{-\Phi(x)}$. Thus, if $\limsup_{x\rightarrow 0} \Phi(x) = 0$, we would have $\varphi(x) = x^{-\Phi(x)}\leq x^{-1/N}$ for small enough $x$, so that 
\[\PP\left\{\mu: \limsup_{m\rightarrow\infty} \frac{\log_2 H_{m,\Phi}(\mu)}{m}<\infty \right\}=0.\]

It might be the case that $\Phi_n(x)$ does not tend to zero as $x\rightarrow 0$, but $\Phi_n\rightarrow \Phi$ pointwise and $\limsup_{x\rightarrow 0} \Phi(x)=0$. Then we may define some ``quasi-$\Phi$" dimensions analogous to the quasi-Assouad dimension, and such dimensions will be $\PP$-almost surely infinite.

To summarize the above discussion
\begin{theorem}
Let $\PP$ be the distribution on $\P([0,1]^d)$ constructed in Section \ref{sec:distn-on-meas}. We have
\begin{enumerate}
\item Let $\Phi: (0,1) \rightarrow \RR^+$ be a  function such that $\varphi(x) = x^{-\Phi(x)}$ is nonincreasing and left-continuous on some neighbourhood of 0. If $\limsup_{x\rightarrow 0} \Phi(x)=0$, then the (upper) $\Phi$-dimension, 
\[\dim_{\Phi}\mu = \limsup_{m\rightarrow\infty} \frac{\log_2 H_{m,\Phi}(\mu)}{m},\]
is $\PP$-almost surely infinite.

\item Let $\Phi_n: (0,1) \rightarrow \RR^+$ be a sequence of functions such that $\varphi_n(x) = x^{-\Phi_n(x)}$ are nonincreasing and left-continuous on some neighbourhood of 0. Suppose also that $\Phi_n \searrow \Phi$ pointwise, and that $\limsup_{x\rightarrow 0} \Phi(x)=0$. Then the \emph{(upper) quasi-$\Phi$ dimension},
\[\dim_{q\Phi}\mu:= \lim_{n\rightarrow\infty} \limsup_{m\rightarrow\infty} \frac{\log_2 H_{m,\Phi_n}(\mu)}{m},\]
is $\PP$-almost surely infinite.
\end{enumerate}

\end{theorem}

\section{Questions}
We end with a list of questions that we have not yet been able to resolve.

\label{sec:ques}
\begin{enumerate}
\item We saw that $\dim_{qA} \mu = \lim_{\delta\rightarrow 0} H(\mu,\delta)$ is $\PP$-a.s. infinite. Is it true that $H(\mu,\delta)$ is almost surely infinite for some $\delta$? If so, is there a critical $\delta$ where $H(\mu,\delta)$ becomes infinite?

\item It is clear that $\lim_{\delta\rightarrow\infty} H(\mu,\delta)$ is greater than the maximum local dimension. Are they almost surely equal? If not, how much do they differ?

\item We obtained necessary conditions for upper $\Phi$-dimensions to be infinite. How sharp is the result? Can we find precise condition for $\dim_{\Phi}$ to be infinite?

\item What is the answer to the analogs for the lower quasi-Assouad dimension of the above questions?

\item We showed that $\dim_{qA} \mu$ is almost surely infinite by introducing a smaller quantity defined by $H_{m,\delta}(\mu)$ (instead of $H^*_{m,\delta}(\mu)$) and showed the smaller quantity infinite. Do these two quantities agree, at least in some almost sure sense?

\item It can be shown that $\PP$-a.e. measures are supported on $[0,1]^d$, so our approach does not give much information about measures without full support. Is there a way to define random measures supported on arbitrary sets?
\end{enumerate}

\section*{Acknowledgements}
The author would like to thank Professors Kathryn Hare and Kevin Hare for their helpful comments and suggestions throughout this USRA project, as well as their careful reading of the drafts, Professors Spiro Karigiannis and Ruxandra Moraru for their supports on evaluating integrals, Zhenyuan Zhang for some interesting discussions about the work of Furstenberg. The author is also grateful to Professor Ram Murty for his inspiring talk in CUMC 2019. This research was supported in part by NSERC grants RGPIN 2016-03719 and 2019-03930.

\label{last page}
\end{document}